\documentclass[10pt]{amsart}
\usepackage[foot]{amsaddr}
\title{Arquímedes y las superficies cuádricas}
\author{Jonathan Taborda Hernández}
\email{taborda50@gmail.com}
\usepackage{amsmath,amssymb,amsthm,graphicx,cancel,pifont,pb-diagram,mathpazo,epstopdf,wasysym,marvosym,mathrsfs,eso-pic,wallpaper,array,arydshln,multicol,etoolbox,morefloats,etex,mathtools,wrapfig}
\usepackage[activeacute,english,greek,spanish]{babel}
\usepackage[square,numbers]{natbib}
\usepackage[labelfont=bf]{caption}
\usepackage[font={small,it}]{caption}
\usepackage[belowskip=-10pt,aboveskip=6pt,labelsep=period]{caption}
\usepackage[LGR,T1]{fontenc}
\usepackage[10pt,GlyphNames,boldLipsian]{teubner}
\usepackage[titletoc,toc,page]{appendix}
\usepackage[scale=0.9]{ccicons}
\renewcommand{\qedsymbol}{$\blacksquare$}
\usepackage{tikz}

\usepackage{hyperref}
\hypersetup{bookmarksopen,bookmarksnumbered,colorlinks,linkcolor=blue,urlcolor=blue}
\newtheorem*{defini}{Definición}
\newtheorem{teor}{Teorema}

\newtheorem*{propo}{Proposición}

\begin{document}\maketitle
\begin{abstract}
Un breve repaso por la historia de las secciones cónicas no estaría completo sin un conteo exhaustivamente tolerable sobre todas las cosas relativas al tema que pueden ser halladas en la extensa obra del sabio Arquímedes. No se posee evidencia contundente de que el genio siracusano escribiera un tratado sobre las \textit{Cónicas} por separado.
\end{abstract}
\selectlanguage{english}
\begin{abstract}
A brief review of the history of the conic sections would not be complete without an exhaustively tolerable account of all the things related to the subject that can be found in the extensive work of the wise Archimedes. There is no strong evidence that the Syracusan genius wrote a treatise on \textit{Conics} separately.
\end{abstract}
\selectlanguage{spanish}
\tableofcontents
\section{Introducción}
Un breve repaso por la historia de las \textit{secciones cónicas} no estaría completo sin un conteo exhaustivamente tolerable sobre todas las cosas relativas al tema que pueden ser halladas en la extensa obra del sabio Arquímedes.\\ No se posee evidencia contundente de que el genio siracusano escribiera un tratado sobre las \textit{cónicas} por separado.\\ La idea de que él lo hizo reposa en efecto sobre una base no muy substancial en la referencia para \textLipsias{konik\A{a} stoiqe\C{i}a} [\textit{elementos de las cónicas}] (sin alguna mención sobre el nombre del autor) en dicho \textit{folium}, que ha sido asumida como una referencia para un tratado del mismo Arquímedes. Pero dicha suposición sencillamente puede ser refutada cuando las referencias son comparadas con una cita similar en otro pasaje\footnote{Cf. \textit{Sobre la Esfera y el Cilindro. I. pp. 24.} en Heiberg. \textit{Archimedis Opera Omnia.} Vol. I. pp. 24. La proposición citada está en \textit{EE}. XII. 2: \textit{Los círculos son el uno al otro como los cuadrados de sus diámetros.}\\ Esta proposición (i.e., \textit{EE} X.I.) debe ser bien recordada porque es el lema requerido en la demostración que efectúa Euclides de la proposición 2, libro XII para el efecto en el que los círculos son el uno al otro como los cuadrados de sus diámetros. Algunos escritores parecen tener la impresión de que XII.2 y las otras proposiciones en el libro XII, en que el método de exhausión  es empleado, son los únicos lugares (\textit{hapax legomena}) donde Euclides hace uso de X.I. Además, es comúnmente visto que X.I. tendría que ser jústamente referida al inicio del libro XII. Cantor (cf. \textit{Gesch. d. Math. $I_3$. p. 269.}) observa: <<Euclides no hace ninguna referencia a este [X.I.], inclusive nosotros no tenemos nada más que esperar algo, a saber que, si dos magnitudes son inconmensurables, nosotros siempre podemos formar una magnitud conmensurable con la primera que será diferente de la segunda magnitud por pequeña que a nosotros nos plazca>>. Pero, sin hacer uso de X.I. antes de XII.2, Euclides emplea esta en toda la proposición siguiente, (X2). Se tiene entonces que X.2 produce un criterio para la inconmensurabilidad de dos magnitudes (un preliminar muy necesario para el estudio de los inconmensurables); X.I. será exactamente este. Euclides emplea X. I. para demostrar no únicamente XII.2 sino XII.5 (i.e., las pirámides con la misma altura y base triangular no son la una a la otra como sus bases), por medio del cual él demuestra XII.7 y Porisma, que alguna pirámide es la tercera parte del prisma que tendrá la misma base e igual altura y XII.10 (i.e., que algún cono es una tercera parte del cilindro que tendrá la misma base e igual altura), entre otras proposiciones similares. Ahora, XII.7, Porisma y XII.10 son teoremas específicamente atribuidos a Eudoxus por Arquímedes (cf. \textit{Sobre la esfera y el cilindro}, prefacio), quien dice en otro lugar (cf. \textit{Cuadratura de la parábola}, prefacio) que el primero de los dos, y los teoremas sobre los círculos que son el uno al otro como los cuadrados de sus diámetros, se demuestra por cierto lema que él establece como sigue: << de líneas desiguales, superficies desiguales, o sólidos desiguales, el más grande excede al menor por tal magnitud como sea posible, si se agrega [continuamente] a sí misma, excedería alguna magnitud de aquellas que son comparadas una con otra>>, i.e., de magnitudes del mismo tipo como las magnitudes originales. Arquímedes también afirma (loc. cit.) que el segundo de los dos teoremas que él atribuye a Eudoxus (EE. XII.10) fueron demostrados por medios de un <<Lema similar mencionado anteriormente>>. El lema establecido así por Arquímedes es decididamente diferente del X.I, el cual sin embargo, Arquímedes mismo empleó en varias ocasiones, mientras él se refiere al uso de éste en XII.2 (cf. \textit{Sobre la esfera y el cilindro}, I.6.). La aparente dificultad causada por la mención de los dos lemas en conexión con el teorema de Euclides XII.2 puede ser explicada por la referencia a la demostración de X.I. Euclides aquí toma la magnitud menor y dice que es imposible, por la multiplicación de esta, que alguna vez exceda las más grande, y claramente emplea el estamento de la cuarta definición del libro V, para el efecto de que <<se dice que las magnitudes están en proporción una con otra, si multiplicándolas, se exceden una a la otra>>. Desde entonces la magnitud más pequeña en X.I puede ser vista como la diferencia entre las dos magnitudes desiguales, y es claro que el lema establecido por Arquímedes es en substancia empleado para demostrar el lema X.I que parece jugar un papel mucho más extenso en las investigaciones sobre la cuadratura y la cubatura.} en el que por las palabras \textLipsias{\A{e}n t\C{h} stoiqei\A{w}dei} [\textit{en sus elementos}] en los \textit{Elementos de Euclides} indudablemente lo dan a entender.\\ De manera similar las palabras <<esto es demostrado en los elementos de las cónicas>> simplemente significan que esto está hecho en los libros de texto sobre los principios elementales de las cónicas. Una demostración positiva es que esto puede ser rastreado desed un pasaje sobre los comentarios de Eutocious a la obra de Apollonius. Heracleides,\footnote{El nombre aparece en el pasaje refiriéndose a este como \textLipsias{\<Hrakleios}. Cf. \textit{Apollonius} (ed. Heiberg) Vol. II. pp. 168.} el biógrafo de Arquímedes, puede ser citado diciendo que <<Arquímedes fue el primero en inventar teoremas sobre las cónicas, y Apollonius había encontrado que ellos no habían sido publicados por Arquímedes, apropiándose de ellos>>;\footnote{La sentencia de Heracleides de que Arquímedes fue el primero en <<inventar>> (\textLipsias{\A{e}pino\C{h}sai}) teoremas sobre las cónicas no es sencillo de explicar. Bretschneider (pp. 156.) afirma, con respecto al cargo de plagio levantado contra Apollonius, bajo la malicia de las pequeñas mentes tendrían probablemente que buscar venganza por ellos mismos para el concepto en que ellos tendrían que ser ayudados por un gigante intelectual similar a Apollonius. Heiberg, por otro lado, piensa que esto es injusto para Heracleides, quien probablemente malinterpretó la elaboración del cargo de plagio por encontrar muchas de las proposiciones de Apollonius ya citadas por Arquímedes, como es bien conocido. Heiberg deja en claro también que Heracleides no tuvo la intención de adscribir la invención actual de las cónicas a Arquímedes, sino únicamente la teoría elemental de las secciones cónicas como fueron formuladas por Apollonius debidas al siracusano; por otra parte, la contradicción de Eutocius tendría que tomar una forma diferente y él no tendría que haber omitido el punto bien conocido de que Menaechmus fue el descubridor de las secciones cónicas. Cf. Sir Thomas L. Heath. \textit{Apollonius of Perga}. Teatrise of Conic Section. Cap. III. Cambrigde Univ. Press. 1896.} y Eutocius se une a la observación alegando que en su opinión esto no es verdadero, <<por un lado Arquímedes aparece en muchos pasajes haciendo referencia a los elementos de las cónicas como un tratado antiguo (\textLipsias{\A{w}c palaiot\A{e}ras}) [conocidos previamente], y por otro lado, Apollonius no parece haber enseñado sus propios descubrimientos>>.\\ Así Eutocius estimó la referencia al inicio de una exposición temprana de la teoría elemental de las cónicas por otros geómetras; por otra parte, i.e., si él tuvo que pensar que Arquímedes se refería a un tratado temprano de su propia autoría, él no podría haber empleado la palabra \textLipsias{palaiot\A{e}ras} [previo] en lugar de alguna expresión similar para \textLipsias{pr\A{o}teron \A{e}kdedom\A{e}nhc} [antiguamente emitido].\\ En la investigación de las variadas proposiciones sobre las cónicas que pueden ser halladas en Arquímedes, es natural un vistazo, en primera instancia, para mostrar cómo el sabio siracusano está al tanto de la posibilidad para producir las tres secciones cónicas de los otros conos, en conos rectos y a través de otras secciones planas de aquellas perpendiculares a un generador del cono. Nosotros observamos, primero, que él siempre emplea los antiguos nombres <<sección de un cono de ángulo-recto>>, etc., empleados por Aristaeus, y no debe dudarse de que aquellos tres lugares donde la palabra \textLipsias{\Ar{e}lleuyic} [elipse] aparece no serán relevantes aquí.\\ Al principio del tratado \textit{Sobre conoides y esferoides} (\textLipsias{Konoidec u Esferoidec}) encontramos lo siguiente: <<si un cono es cortado por un plano encontrándose contenido en el cono, la sección será un círculo o una sección de un cono de ángulo-agudo>> (i.e., una elipse). La dirección en que tales proposiciones fueron demostradas en el caso en que la sección del plano forma un ángulo recto con el plano de simetría, puede ser inferida de las proposiciones 7 y 8 del mismo tratado, donde se muestra que es posible encontrar un cono del que una elipse es una sección y cuyo vértice no es una línea recta trazada desde el centro de la elipse (1) perpendicular al plano de la elipse, (2) no es perpendicular a este plano, pero reposa en un plano en ángulo recto a este y pasa a través de uno de los ejes de la elipse. El problema evidentemente pertenece a la determinación de las secciones circulares del cono, y así es como Arquímedes procede:
\begin{enumerate}
\item[(1)] Conciba una elipse con $BB'$ su eje menor y levante un plano perpendicular al plano del papel: suponga que la línea $CO$ traza una perpendicular al plano de la elipse, y sea $O$ el vértice del cono requerido, (véase Fig.\ref{B.1}).
\begin{figure}[ht!]
\begin{center}
\includegraphics[scale=0.5]{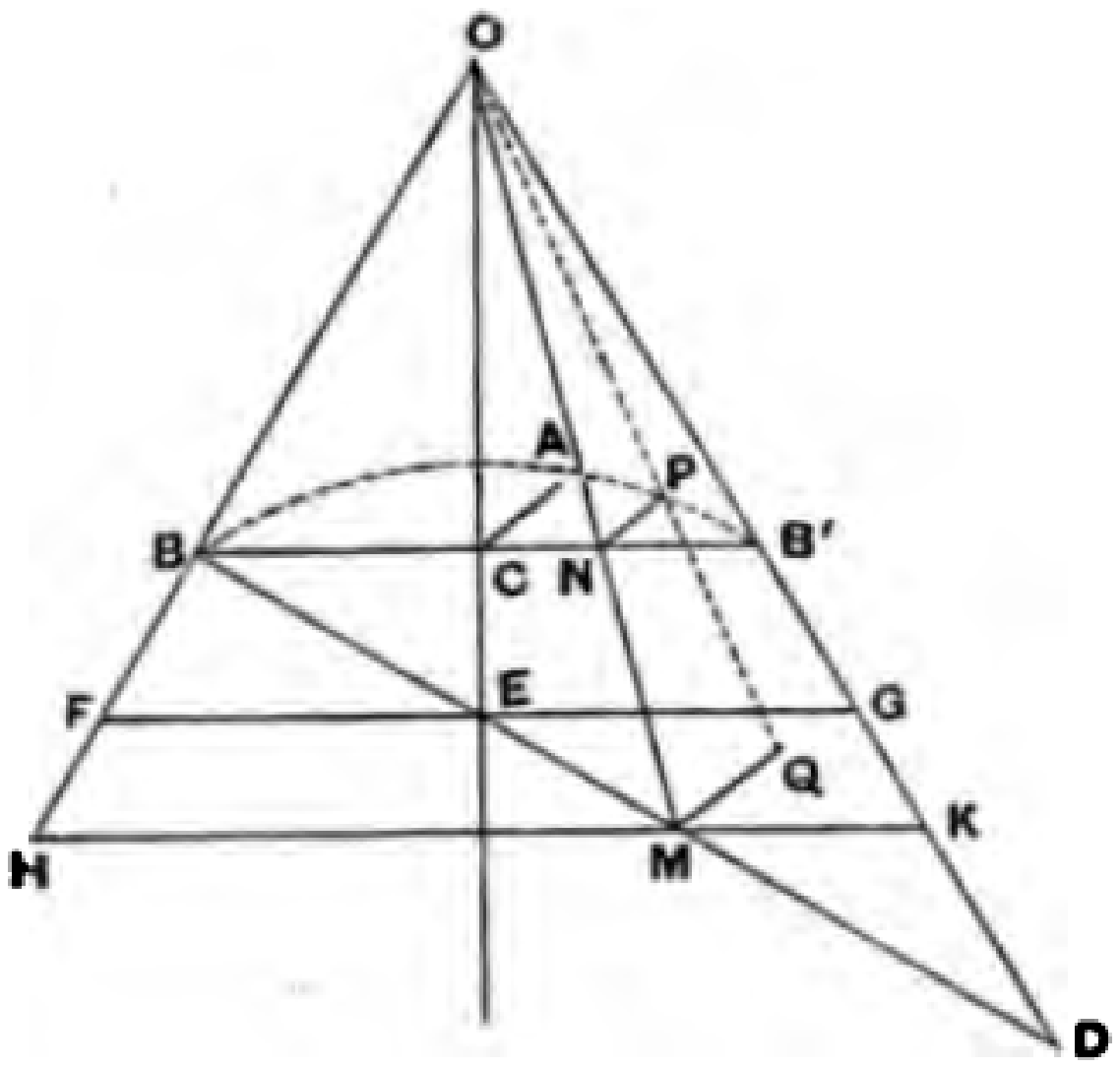}\\
\caption{}\label{B.1}
\end{center}
\end{figure}
\noindent
Tómese $OB, OC, OB',$ y en el mismo plano con ellos trace $BED$ encontrando $OC, OB'$ en $E,D$ respectivamente, y en tal dirección se tiene $$BE\cdot ED=EO^2=CA^2\cdot CO^2$$ (donde $CA$ es la mitad del eje mayor de la elipse.) Y si esto es posible, entonces $$BE\cdot ED:EO^2>BC\cdot CB':CO^2$$ (Ambos, la construcción y ésta última proposición son asumidos como conocidas).\\ Ahora conciba un círculo con $BD$ como diámetro trazado en un plano perpendicular al plano del papel, y describa un cono pasando a través del círculo y tome a $O$ como su vértice.\\ Tenemos que probar entonces que dada una elipse, esta es una sección de dicho cono, o, si $P$ es algún punto sobre la elipse, $P$ está sobre la superficie del cono.\\ Trace $PN$ perpendicular a $BB'$. Una $ON$, y prolongue esta hasta encontrar a $BD$ en $M$, y sea $MQ$ trazada en el plano del círculo sobre $BD$ como diámetro y perpendicular a $BD$, encontrando la circunferencia del círculo en $Q$. También trácese $FG, HK$ a través de $E,M$, respectivamente paralelos a $BB'$.\\ Ahora
\begin{align*}
QM^2:HM\cdot MK&=BM\cdot MD:HM\cdot MK\\
&=BE\cdot ED:FE\cdot EG\\
&=\left(BE\cdot ED:EO^2\right)\cdot\left(EO^2:FE\cdot EG\right)\\
&=\left(CA^2:CO^2\right)\cdot\left(CO^2:BC\cdot CB'\right)\\
&=CA^2:BC\cdot CB'\\
&=PN^2:BN\cdot NB\\
\therefore\,QM^2:PN^2&=HM\cdot MK: BN\cdot NB'\\
&=OM^2:ON^2,
\end{align*}
donde, $PN, QM$ son paralelos, $OPQ$ es una línea recta.\\ Pero $Q$ no es la circunferencia del círculo sobre $BD$ como su diámetro; entonces $OQ$ es un generador del cono, y entonces $P$ está sobre el cono.\\ Así el cono pasa a través de todos los puntos de la elipse dada.\\
\item[(2)] Sea $OC$ no perpendicular a $AA'$, uno de los ejes de la elipse dada, y el plano del papel está conteniendo a $AA'$ y $OC'$, así que el plano de la elipse es perpendicular al plano. Sea $BB'$ el otro eje de la elipse, (véase Fig.\ref{B.2}).
\begin{figure}[ht!]
\begin{center}
\includegraphics[scale=0.5]{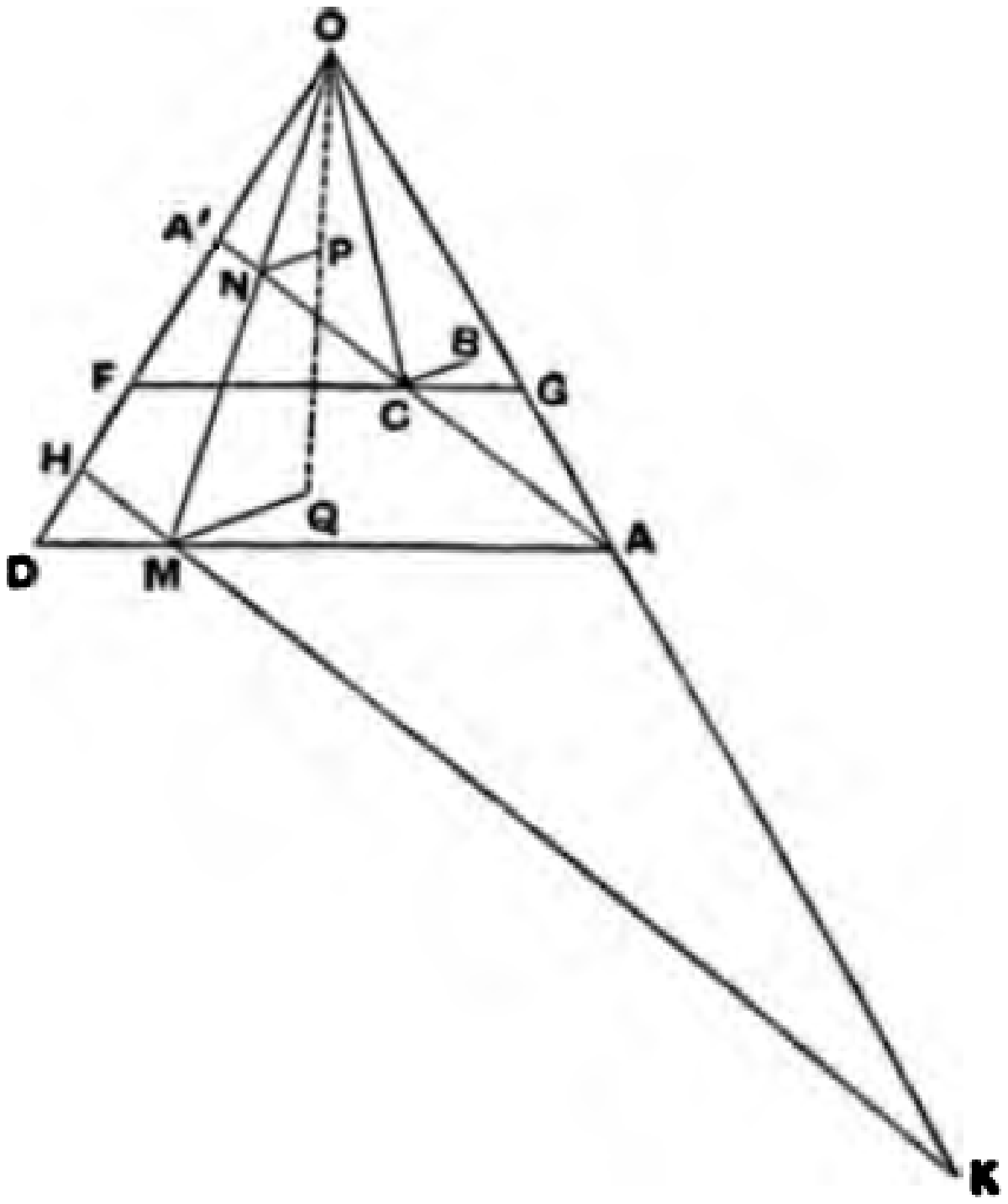}\\
\caption{}\label{B.2}
\end{center}
\end{figure}
\newline
Ahora $OA, OA'$ no son iguales. Prolóngese $OA'$ hasta $D$ así que $OA=OD$. Una $AD$, y trace $FG$ a través de $C$ paralelo a este.\\ Conciba un plano a través de $AD$ perpendicular al plano del papel, y en este describa\\
\item[(a)] si $CB^2=FC\cdot CG$, un círculo con diámetro $AD$, o\\
\item[(b)] si no, una elipse sobre $AD$ como eje tal que si $d$ es el otro eje $$d^2: AD^2=CB^2:FC\cdot CB.$$
\end{enumerate}
\par
Tome un cono con vértice en $O$ pasando a través del círculo o una elipse justamente trazada.\\ Esto es posible inclusive cuando la curva sea una elipse, porque la línea trazada desde $O$ al punto medio de $AD$ es perpendicular al plano de la elipse, y la construcción es similar a la precedente en el caso (1).\\ Sea $P$ algún punto sobre la elipse, y nosotros únicamente tenemos que probar que $P$ está sobre la superficie del cono así descrito.\\ Tracemos $PN$ perpendicular a $AA'$. Unamos $ON$, y prolonge esta para encontrar $AD$ en $M$. A través de $M$ trace $HK$ paralela a $AA'$. Finalmente, trace $MQ$ perpendicular al plano del papel (y entonces perpendicular a $HK$ y $AD$) encontrando la elipse o círculo sobre $AD$ (y entonces la superficie del cono) en $Q$.\\ Entonces
\begin{align*}
QM^2:HM\cdot MK&=\left(QM^2:DM\cdot MA\right)\cdot\left(DM\cdot MA:HM\cdot MK\right)\\
&=\left(d^2: AD^2\right)\cdot\left(FC\cdot CG:A'C\cdot CA\right)\\
&=\left(CB^2:FC\cdot CG\right)\cdot\left(FC\cdot CG: A'C\cdot CA\right)\\
&=CB^2: A'C\cdot CA\\
&=PN^2:A'N\cdot NA\\
\therefore\, QM^2:PN^2&=HM\cdot MK:A'N\cdot NA\\
&=OM^2:ON^2
\end{align*}
Por tanto $OPQ$ es una línea recta, y, $Q$ está sobre la superficie del cono, y esto se sigue de que $P$ está también sobre la superficie del cono.\\ La demostración de que las tres cónicas pueden ser producidas por medio de secciones de algún cono circular, recto u oblicuo, que son hechas por planos perpendiculares al plano de simetría, pero no necesariamente perpendicular a las líneas generatrices del cono, es de hecho esencialmente la misma que la demostración para la elipse.\\ Debe entonces inferirse que Arquímedes estaba igualmente al tanto sobre el hecho de que la parábola y la hipérbola podrían ser halladas por métodos más antiguos. Él continua empleando los nombres antiguos para las curvas y en dicho contexto la elipse fue llamada <<sección del cono de ángulo agudo>>, despues fue descubierto que esta podría ser producida por medio de un plano que las corte todas desde de un cono, siempre y cuando lo haga en ángulo recto. Heiberg concluye que Arquímedes únicamente obtuvo la parábola empleando los antiguos métodos porque él describe el parámetro como el doble de la línea entre el vértice de la parábola y el eje del cono, que es únicamente correcto en el caso del cono de ángulo recto; pero esto no es más que una objeción para el uso continuo del término arquimediano <<sección de un cono de ángulo agudo>> para significar que la elipse fue hallada de forma diferente. Zeuthen señala, como evidencia posterior, el hecho de que nosotros tenemos las siguientes proposiciones enunciadas por el siracusano sin demostración (\textit{Sobre conoides y esferoides}, 11):
\begin{enumerate}
\item[(1)] Si un conoide de ángulo recto [un paraboloide de revolución] es cortado por un plano a través del eje o paralelo al eje, la sección será la de
un cono de ángulo recto y la misma que comprende la figura (\textLipsias{\As{a} a\G{u}t\G{a} t\C{a} perilambano\A{u}sa t\G{o} sq\C{h}ma}) [aquellas figuras que incluyen la forma]. Y su diámetro [eje] será la sección común del plano que corta la figura y que es trazada a través del eje perpendicular al plano de corte.\\
\item[(2)] Si un conoide de ángulo obtuso [hiperboloide de revolución] es cortado por un plano a través del eje o paralelo al eje o a través del cono
envolvente (\textLipsias{peri\A{e}qontoc}) [que contiene] el conoide, la sección será una sección de un cono de ángulo obtuso: si [el plano de corte pasa] a través del eje, el mismo que está comprendido en la figura: si es paralelo al eje, o similar a este; y si es a través del vértice del cono envolvente del conoide, es no similar. Y el diámetro [eje] de la sección será una sección común para el plano que corta la figura y de la que se traza el eje en ángulos rectos al plano de corte.\\
\item[(3)] Si alguna de las figuras esferoidales son cortadas por un plano a través del eje o paralelo al eje, la sección será una sección de un cono de
ángulo agudo: si es a través del eje, la sección actual que se obtiene es la que comprende la figura: si es paralelo al eje, es similar a este.
\end{enumerate}
\par
Arquímedes agrega que las demostraciones de todas aquellas proposiciones son obvias. Sí es entonces tolerablemente cierto que ellas están basadas sobre los mismos principios como sus demostraciones tempranas relativas a las secciones de una superficie cónica y las demostraciones dadas en sus investigaciones posteriores sobre las secciones elípticas para las variadas superficies de revolución. Aquellas dependen, como se verá, de la proposición que afirma, que si dos cuerdas trazadas en direcciones fijas intersectan un punto, la proporción de los ángulos rectos bajo los segmentos es independiente de la posición del punto. Esta corresponde exactamente a la empleada en las demostraciones anteriores sobre el cono, para la proposición en que, si las líneas rectas $FG, HK$ son trazadas en direcciones fijas entre dos líneas formando un ángulo, y si $FG,HK$ se encuentran en algún punto $M$, la proporción $FM\cdot MG:HM\cdot MK$ es constante; la propiedad posterior es en efecto el caso particular en que la cónica se reduce a dos líneas rectas.\\ La siguiente es una reproducción, dada para un ejemplo, de la proposición (13)\footnote{Cf. Heiberg. Op. cit. Vol.I. pp. 348 y ss.} del tratado \textit{Sobre conoides y esferoides}, que prueba que la sección de un conoide de ángulo obtuso [un hiperboloide de revolución] por algún plano que encuentra todos los generadores del cono envolvente, y que no es perpendicular al eje, es una elipse cuyo eje mayor es la parte interceptada en el interior del hiperboloide de la línea de intersección para el plano de corte y el plano a través del eje perpendicular a este.
\section{Los elementos de trabajo de Arquímedes}
Propiamente hablando sobre los \textit{Elementos de Euclides}, los matemáticos griegos estaban en posesión de una colección ordenada sistemáticamente de proposiciones matemáticas\footnote{Para un estudio detallado y pormenorizado de los recursos lingüísticos empleados por los matemáticos griegos cuando efectúan la demostración de un resultado, véase: Fabio Acerbi. \textit{La sintassi logico della matematica greca.} Libro I.: \url{https://hal.archives-ouvertes.fr/hal-00727063/file/LibroLogica1.pdf}., versión 1. 4 sep. 2012.\\ \textit{The language of the <<Givens>>: its forms and it is use deductive tool in Greek mathematics.} Arch. Hist. Exac. Sci. (2011) 65: 119-153.} sobre las que ellos tendrían que basar sus investigaciones futuras. De esta manera se exhoneraban de la obligación de revertir en sus trabajos los temas de contenido elemental, i.e., de naturaleza fundamental: cuando una proposición estaba contenida en los \textit{Elementos}, esto era suficiente, para el aparato de difusión universal del trabajo, y era una razón suficiente para no mencionar esta.\\ El grado de conocimiento matemático requerido en la lectura moderna de los trabajos de Arquímedes constituye infrecuentemente una seria barrera para la apreciación de ellos, porque él adquiere su conocimiento de los elementos de su ciencia de una manera enteramente diferente de los discípulos de Euclides. Además, Euclides acostumbra a preceder el núcleo real de sus tratados por numerosos lemas, cuyo propósito no será aparentemente útil hasta mucho después.\footnote{Cf. \textit{Archimedes} E.J. Dijksterhuis. transladed by C. Dikshoorn with a new bibliographic essay by Wilbur R. Knorr. Princeton Univ. Press. 1987. pp. 49 y ss.}\\ Continuando con la referencia a la obra del Magno Arquímedes, en la carta nuncupatoria a \textit{Dositheus}\footnote{Cf. \textit{Sobre conoides y esferoides.} Introducción. En: \textit{The works of Archimedes} Edit in modern notation with introductory chapters edited by Sir Thomas L. Heath. pp. 99-151. Cambridge University Press. This digitally printed version 2010.} que sirve como exordio al tratado \textit{Sobre conoides y esferoides}, las siguientes figuras son definidas:
\begin{enumerate}
\item[(I)] Cuando un \textit{ortotoma}\footnote{Es un término antiguo, ya en desuso, utilizado por los antiguos griegos cuando todavía las curvas cónicas no tenían un nombre específico, definiéndose cada una por la forma en que habían sido descubiertas. Así se obtenían oxitomas o secciones de un cono agudo, ortomas o secciones de un cono rectángulo y amblitomas o secciones de un cono obtuso. Arquímedes utilizó estos nombres, aunque según parece, también usó ya el nombre de \textit{parábola}, como sinónimo para una sección de un cono rectángulo. Pero fue realmente Apollonius, posiblemente siguiendo una sugerencia de Arquímedes, quién introdujo por primera vez los nombres de \textit{elipse} y de \textit{hipérbola} en conexión con estas curvas.} (\textLipsias{\>{o}rtotwma}) rota sobre el diámetro, una figura es generada y es llamada un \textit{conoide de ángulo recto} (\textLipsias{\>{o}rjog\A{w}niou qwnoeid\A{e}c}); que puede ser trasladado como (\textLipsias{\>{o}rtoconoide}). Cuando paralelo a algún plano tangente a esta se traza un plano que la corta, este plano junto con la superficie determinan un \textit{segmento de conoide} (\textLipsias{tm\C{a}ma to\C{u} qwnoeid\A{e}oc}), del que la base es la sección para el corte del plano con el conoide; el vértice es el punto de contacto con el plano tangente, el eje la parte de corte externa por los planos de la línea recta a través del vértice del segmento paralelo al eje de revolución.\\ El ortoconoide es aparentemente un \textit{paraboloide de revolución}.\\
\item[(II)] A través de la rotación de un \textLipsias{\>{a}mblutome} sobre el diámetro de un conoide de ángulo obtuso (\textLipsias{\>{a}mblug\A{w}nion qwnoeid\A{e}c}), que puede ser trasladado como \textit{ambliconoide}, se obtiene entonces un \textit{hiperboloide de revolución de una o dos hojas}. Las asíntotas de la sección durante la rotación generan el cono envolvente (\textLipsias{q\A{w}nos peri\A{e}qwn}). De un segmento como el definido en (I) el eje es la parte de corte externa por el plano tangente paralelo a la línea recta que une el vértice del cono envolvente con el vértice del segmento. La parte de la línea recta entre aquellos puntos por si mismos es llamado \textit{eje producido} (\textLipsias{poteo\C{u}sa t\C{w} \As{a}xoni}\,$=\text{adyacente al eje}$).\\
\item[(III)] A través de la rotación del \textit{oxit\={o}ma} (\textLipsias{\>{o}xitwma}) sobre un diámetro se genera el \textit{esferoide} (elipsoide de revolución), que es llamado \textit{prolato} (\textLipsias{param\C{a}qec sfairoid\A{e}c}) u \textit{oblato} (\textLipsias{\G{e}piplat\G{u} sfairoeid\A{e}c}), de acuerdo a la sección que ha rotado sobre el diámetro mayor o menor.\\ El eje y el vértice son definidos en forma similar a la anteriormente descrita, el centro es el centro de la sección rotante, el diámetro de la sección son los ángulos rectos al eje de rotación. Un plano para el corte junto con el esferoide determina dos segmentos, los vértices que son los puntos donde los planos tangentes paralelos al plano de corte tocan la superficie, las partes determinadas por el plano de corte sobre el segmento de línea entre los puntos de contacto son llamados los \textit{ejes}.
\end{enumerate}
\section{El ortoconoide (paraboloide de revolución)}
\begin{propo}[C.S. 11(1)\footnote{Cf. DE CONOIDIBUS ET SPHAEROIDIBUS pp. 341 y ff. en \textit{Archimedis Opera Omnia Vol.I. CUM COMMENTARIIS EUTOCII.} J. L. Heiberg. B. G. TEUBNERL. 1880.}]
Si un \textit{ortoconoide} es cortado a través de un plano, o paralelo al eje, la sección será el mismo \textit{ortotoma} que está contenido en la figura que revoluciona, y su diámetro será la sección común del plano de intersección de la figura y el plano levantado a través del eje en ángulos al plano de corte.
\end{propo}
\begin{proof}
La demostración, que está perdida en el tratado de Arquímedes (esta es requerida únicamente para un plano paralelo al eje), puede ser hallada como sigue (ver Fig. \ref{B.3})\footnote{Cf. Dijksterhuis. op.cit. pp.113 y ss.}
\begin{figure}[ht!]
\begin{center}
\includegraphics[scale=0.5]{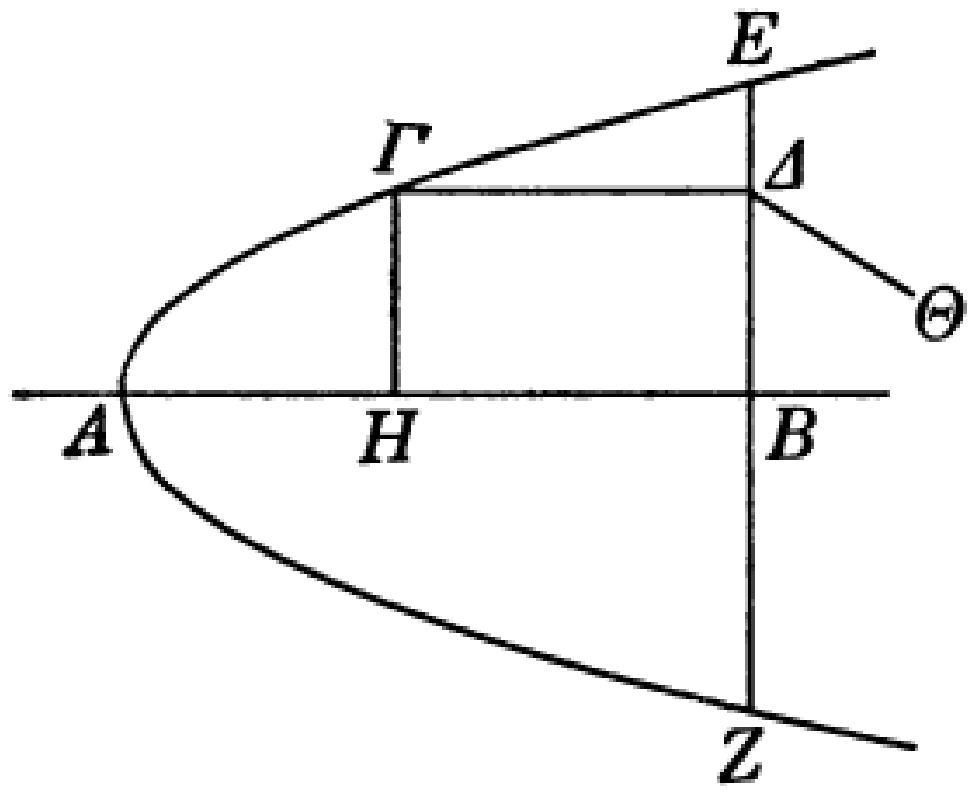}\\
\caption{}\label{B.3}
\end{center}
\end{figure}
\newline
Sea $AB$ el eje, \textLipsias{GD} la intersección entre el plano del papel y el plano de corte en ángulos rectos. Se traza una línea recta a través de \textLipsias{D} perpendicular al plano del papel (i.e., en el plano de corte) encontrando la superficie en \textLipsias{Q}. Entonces, si el orto para la sección meridiana es $N$, tenemos
\begin{align*}
T(\Theta\Delta)&=O(E\Delta,Z\Delta)=T(BE)-T(B\Delta)\\
&=O(N,AB)-T(\Gamma H)=O(N,AB)-O(N,AH)\\
&=O(N,BH)=O(N,\Gamma\Delta)
\end{align*}
La sección entonces es un ortotoma con vértice en $\Gamma$, diámetro $\Gamma\Delta$, y el mismo orto como $EAZ$, i.e., iguales y similares.
\end{proof}
\begin{propo}[C.S.12\footnote{Cf. Heiberg Op.cit. pp. 345-349.}] Si un ortoconoide es cortado por un plano que pasa a través del eje, no paralelo al eje, no perpendicular al eje, la sección será un oxitoma; el diámetro más grande será la parte de corte en el interior del conoide de la intersección entre el plano de corte y el plano a través del eje en ángulos rectos al plano de corte, el diámetro menor será igual a la distancia entre las líneas rectas trazadas desde las extremidades del diámetro mayor, paralelo al eje.
\end{propo}
\begin{proof}
En la figura \ref{B.4} sea $A\Gamma$ la intersección entre el plano meridiano, que es el plano del papel, y el plano de corte en ángulo recto, sea $K$ un punto de la intersección entre el plano de corte y la figura, $K\Theta$ la perpendicular de $K$ a $A\Gamma, EZ$ la intersección entre el plano del papel y un plano levantado a través de $\Theta$ en ángulo recto al eje; este plano intersecta al conoide en un círculo sobre $EZ$ como diámetro.
\begin{figure}[ht!]
\begin{center}
  \includegraphics[scale=0.5]{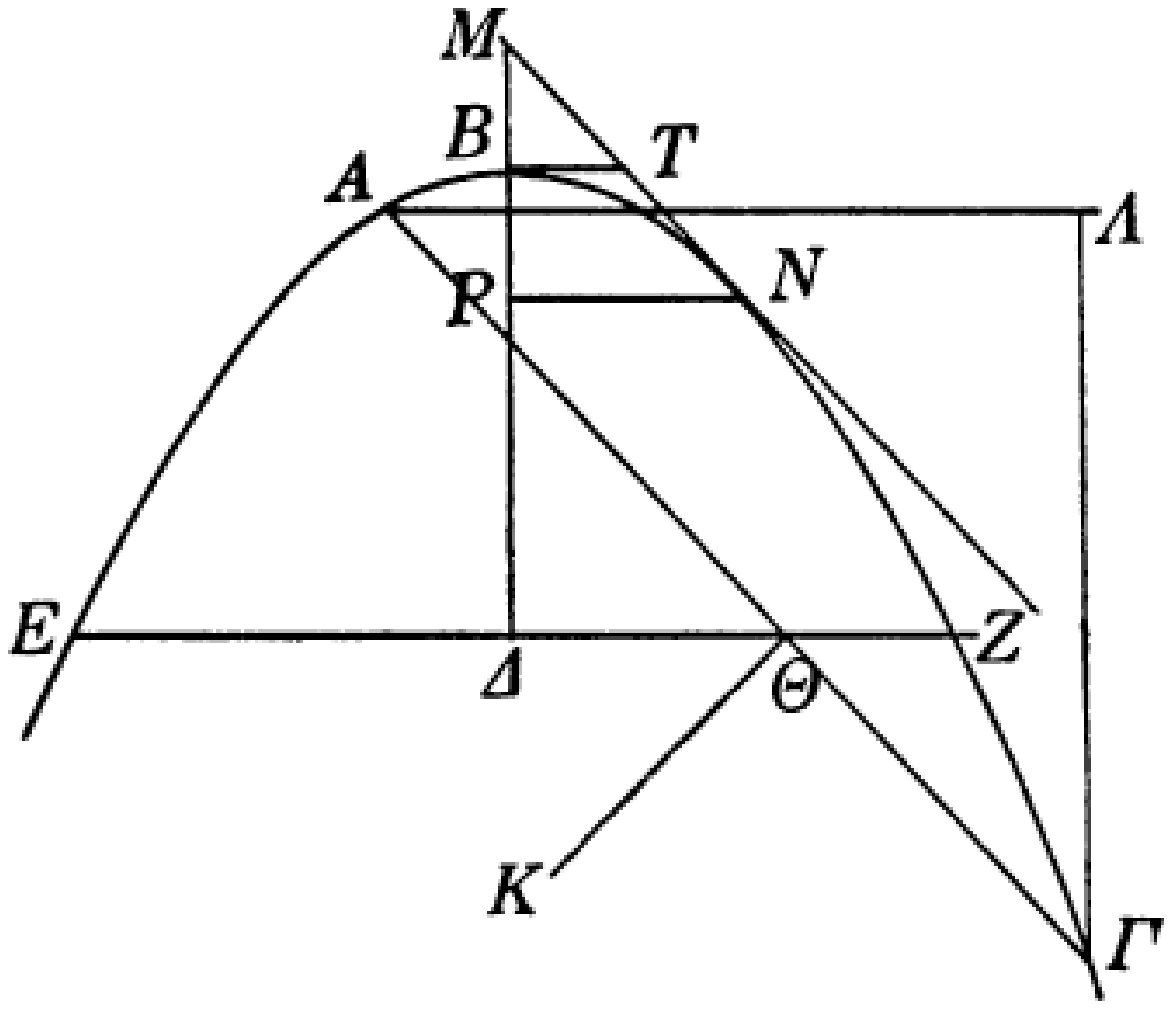}\\
  \caption{}\label{B.4}
\end{center}
\end{figure}
\newline
Ahora $$T(K\Theta)=O(E\Theta, Z\Theta).$$ Ahora, empleando la proposición 3 de \textit{conoides y esferoides}, se tiene que $$[O(E\Theta, Z\Theta), O(A\Theta, \Gamma\Theta)]=[T(BT),T(NT)]$$ cuando la tangente $NM$ paralela a $A\Gamma$ encuentra la tangente al vértice (paralelo a $EZ$) en $T$ porque $NT=MT$ (que se sigue de $PB=BM$) se tiene que $$[T(K\Theta), O(A\Theta, \Gamma\Theta)]=[T(BT), T(MT)]=[T(A\Lambda), T(A\Gamma)]$$ entonces $A\Lambda$ y $A\Gamma$ son respectivamente el diámetro mayor y menor de la oxitoma que es el lugar de \textLipsias{K}.
\end{proof}
\begin{propo}[C.S. 15(1)\footnote{Cf. Heiberg. Op.cit. pp. 357.}] De las líneas rectas, trazadas desde los puntos de un ortoconoide paralelo al eje, las partes que están en la misma dirección de la convexidad (\textLipsias{taqurta}) a la superficie caerán por fuera del conoide; las partes que están en otra dirección caerán en su interior.\\ Este resultado se obtiene con la ayuda de la Proposición 11 (1) correspondiente sobre el ortotoma.\footnote{\textbf{Proposición} [C.S. 11(1)] \textit{Si un paraboloide de revolución es cortado por un plano, o paralelo a este en los ejes, la sección será una parábola igual a la original que por su revolución genera el paraboloide. Y los ejes de la sección serán la intersección entre el plano de corte y el plano a través del eje del paraboloide en ángulos rectos al plano de corte.\\ Si el paraboloide es cortado por un plano en ángulo recto a los ejes, la sección será un círculo cuyo centro está sobre el eje}.\\ Cf. Heath. Op.cit. pp.122 y ss., Heiberg. Op. cit. pp. 341.}
\end{propo}
\begin{defini}
Todos los ortoconoides (paraboloides de revolución) son similares.
\end{defini}
Lo que esto significa no es mencionado posteriormente. Probablemente la referencia es simple debido a la similaridad de las secciones generadas por los ortoconoides, que ya se sabe que son similares.
\begin{propo}
Todos los ortotomas son similares uno al otro.
\end{propo}
\begin{proof}
Sobre la rigidez de la definición anterior de similaridad entre dos cónicas, esta hipótesis puede ser probada como sigue:\\
De dos ortotomas sean las ordenadas y las abscisas sucesivamente $y$ y $x,\eta$ y \textLipsias{x}, el ortoide $N$ y $M$.\\ Entonces tenemos
\begin{align*}
T(y)&=O(N,x)\\
T(\eta)&=O(M,\xi)
\end{align*}
Ahora establezcamos entre las abscisas de las dos curvas la relación $$(x,\xi)=(N,M),$$
\begin{align*}
[T(y),T(\eta)]&=[O(N,x), O(M,\xi)]=[T(x), T(\xi)],\\
\shortintertext{entonces}
(y,\eta)&=(x,\xi)\quad\text{o}\quad(y,x)=(\eta,\xi)
\end{align*}
\end{proof}
Se obtiene la siguiente
\begin{propo}
Dos segmentos de ortotomas son llamados similares cuando las bases están en la misma proporción a sus diámetros, mientras en ambos el ángulo entre el diámetro y la base son el mismo.
\end{propo}
En el corpus arquimediano, el siracusano habla repetidamente sobre la similaridad en conexión con las cónicas. Con toda probabilidad él comprendió dicho significado a través del trabajo de Apollonius, cuya concepción está enraízada con el punto de vista de Euclides sobre el mismo tema; su definición (i.e., la de Arquímedes) puede ser reproducida, en forma abreviada, como sigue:
\begin{quote}
<<\textit{Nosotros decimos que dos cónicas son similares cuando las líneas trazadas hacia el diámetro en la dirección de una ordenada son proporcionales a las partes del diámetro que ellos determinan desde el vértice}>>.\footnote{No nos es bastante claro el por qué el Prof. Dijksterhuis, en su obra sobre \textit{Achimedes} realiza la siguiente cita en lo referente a la igualdad y similaridad de las cónicas en el tratado de Apollonius: Cf. \textit{Apollonius} Conica VI. Def. 2; referencia cruzada que pertenece a la monumental edición del Prof. Heiberg: \textit{Apollonii Pergaei quae graece exstant, cum commentariis antiquis} edidit et latine interpretatus est J.L. Heiberg. 2 vol. Leipzig, B.G. Teubner 1891-93. Dicha \textit{Opera}, solo contiene los libros I-IV, y la referencia de Dijksterhuis pertenece al libro VI. La alusión correcta se puede encontrar en, \textit{Apollonius of Perga. Treatise on conic sections. Edited in modern notation.} Sir Thomas L. Heath. Cambridge Univ. Press. 1896. pp. 379.}
\end{quote}
\begin{propo}[C.S. 15 (3)\footnote{Cf. Heiberg. \textit{Archimedes Opera Omnia.} Vol.I. pp. 358-363.}]
Si un plano encuentra un conoide sin cortarlo, éste lo encontrará en un único punto, y el plano trazado a través del punto de contacto y el eje lo harán en ángulo recto con el plano que tocó éste.
\end{propo}
\begin{proof}
Supóngase que el plano toca la superficie en dos puntos \textLipsias{A} y \textLipsias{B}. Trace a través de cada uno de ellos aquellos puntos una línea recta paralela al eje, y trace un plano a través de aquellas dos líneas rectas. Este plano corta el conoide en un ortotoma sobre el que están los puntos \textLipsias{A} y \textLipsias{B}. Los puntos del segmento de línea \textLipsias{AB} entonces caen en el interior de la sección, i.e., en el interior del conoide.\\ La segunda parte de la proposición es evidente para el plano tangente al vértice. En efecto, las tangentes a los vértices a dos secciones de los conoides en los planos a través de los ejes son perpendiculares al eje, y consecuentemente así también para el plano tangente (aquí entonces este es tomado en consideración para que el plano tangente sea determinado por dos tangentes a las curvas sobre la superficie a través del punto bajo cuestión). Si el plano toca al conoide en otro punto, se sobrentiende que el plano contiene la tangente al círculo paralelo y es consecuentemente un ángulo recto a la sección meridiana a través del punto de contacto.
\end{proof}
\section{El ambliconoide (hiperboloide de revolución)}
\begin{propo}[C.S. 11 (2)\footnote{Cf. Heiberg. Op. cit. pp. 341-342.}] Si un ambliconoide (hiperboloide de revolución) es cortado por un plano a través del eje, o paralelo al eje, o a través del vértice del cono que envuelve al conoide, la sección será un amblitoma (\textLipsias{\>{a}mblitoma}), i.e., el plano que pasa a través del eje, el mismo que contiene la figura <rotante>; si ésta es paralelo al eje, uno similar a ella; si esta pasa a través del vértice del cono que envuelve al conoide, uno que no es similar a ella; y el diámetro de la sección será la sección común al plano de corte de la superficie y el plano a través del eje en ángulos rectos al plano de corte.\footnote{La propoposición anterior, consta de cuatro ítems, todos ellos <<evidentes>>, como afirma Heath. Cf. Heath. Op. cit. pp. 122.}
\end{propo}
Dijksterhuis\footnote{Cf. Dijksterhuis. \textit{Archimedes}. pp. 113 y ss.}, ofrece la siguiente demostración:
\begin{proof}
\textbf{11(1)}. La Proposición es evidente para un plano a través del eje.\\
\textbf{11(2)}. Como puede verse en la siguiente figura:
\begin{figure}[ht!]
\begin{center}
  \includegraphics[scale=0.5]{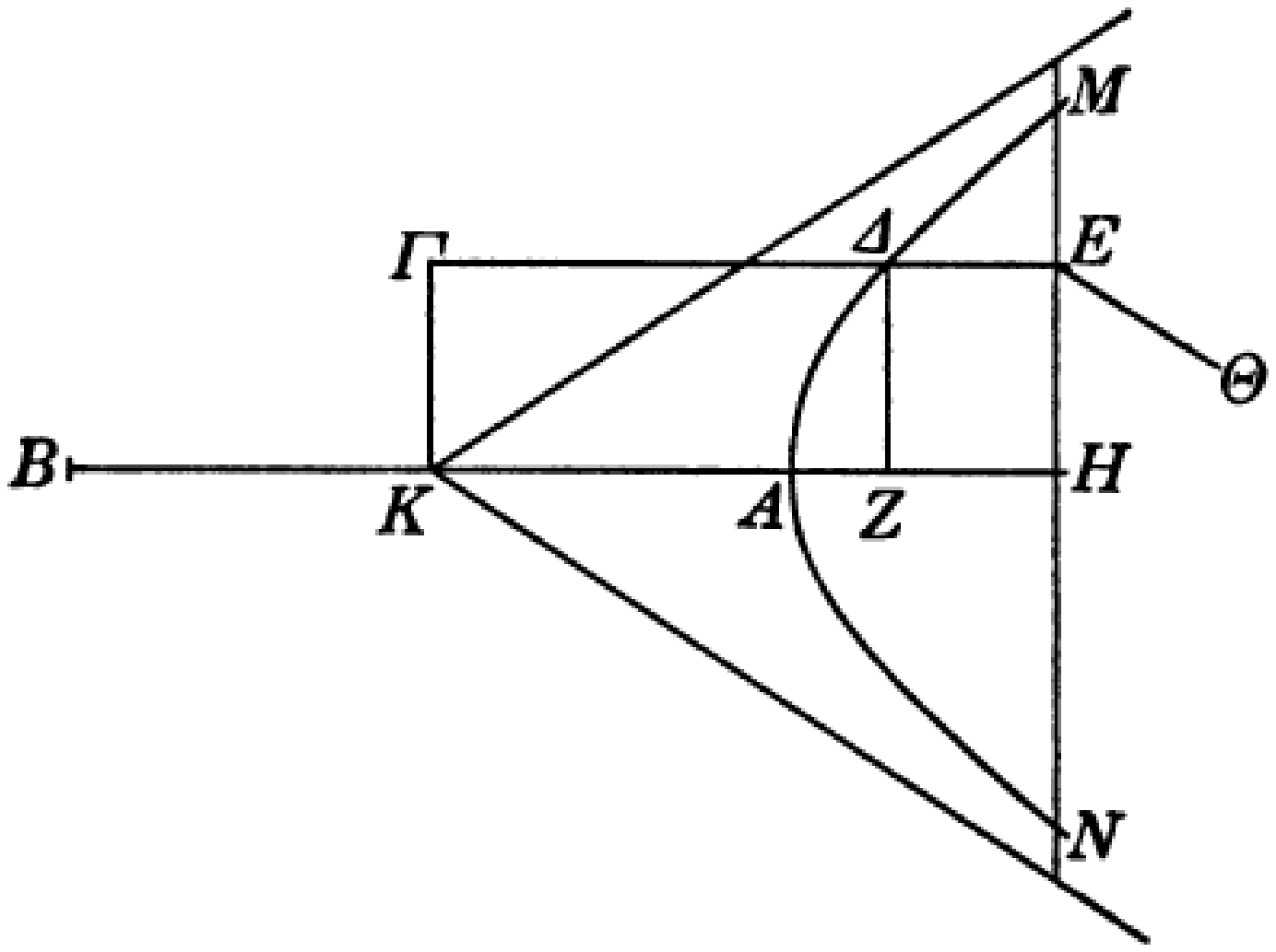}\\
  \caption{}\label{B.5}
\end{center}
\end{figure}
\newline
Sea\textLipsias{GA} la intersección de un plano paralelo al eje \textLipsias{BA} con un plano meridiano elegido como el plano del papel, un ángulo recto al plano, \textLipsias{Q} algún punto de la sección, \textLipsias{QE} la perpendicular de \textLipsias{Q} a \textLipsias{GD}. Ahora, $$T(\Theta E)=O(ME,NE)=T(HM)-T(HE)$$ teniéndose que
\begin{align*}
[T(HM),O(AH,BH)]&=[T(\Delta Z), O(AZ,BZ)],
\shortintertext{entonces}
[T(HM), T(HE)]&=[T(KH)-T(KA), T(KZ)-T(KA)]\\
\shortintertext{o}
[T(\Theta E), T(HM)]&=[T(KH)-T(KZ), T(KH)-T(KA)],\\
\shortintertext{entonces}
[T(\Theta E), T(\Gamma E)-T(\Gamma\Delta)]&=[T(HM), T(KH)-T(KA)]
\end{align*}
El lugar de \textLipsias{J} entonces es un amblitoma que es similar a la sección meridiana.\\
\textbf{11(3)}. En la figura \ref{B.6}
\begin{figure}[ht!]
\begin{center}
  \includegraphics[scale=0.5]{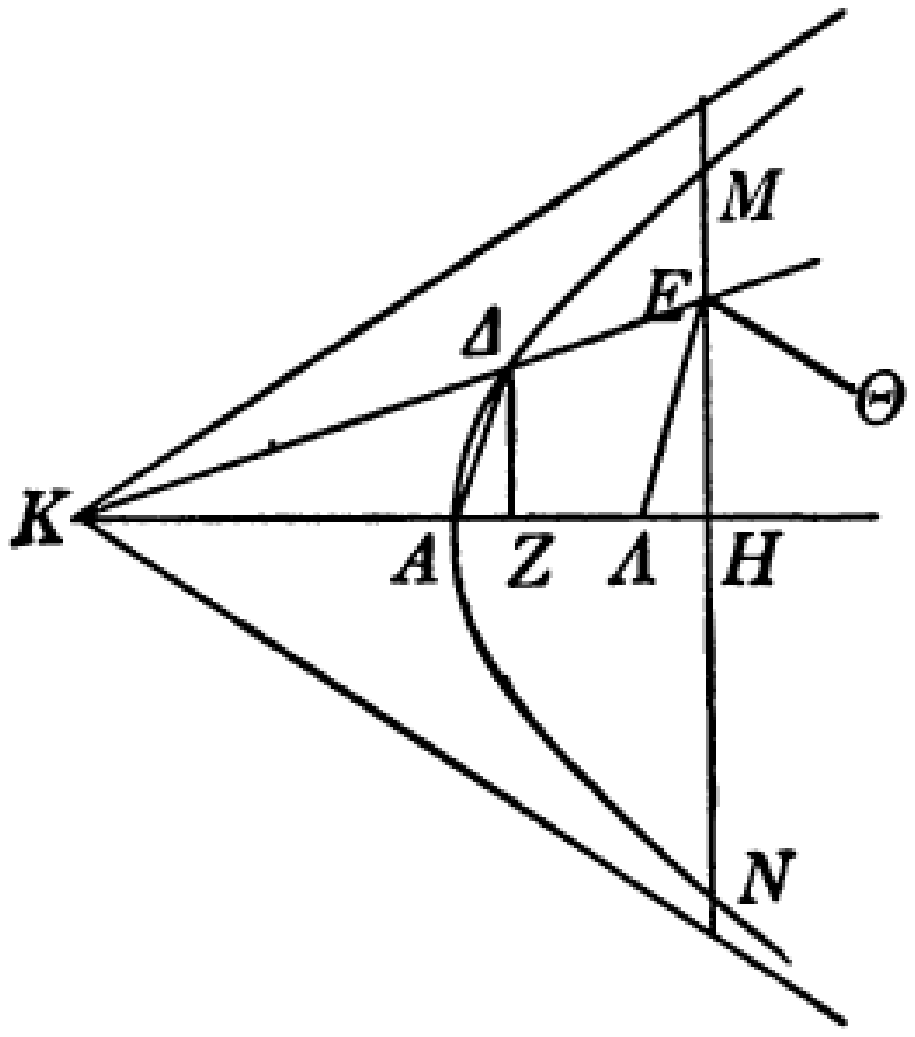}\\
  \caption{}\label{B.6}
\end{center}
\end{figure}
\newline
Sea \textLipsias{KDE} la intersección con el plano de corte a través del vértice para el cono envolvente (centro de la sección meridiana). Ahora tenemos:
$$T(\Theta E)=O(ME,NE)=T(HM)-T(HE)$$ También, si $E\Lambda\parallel\Delta A$,\\
\begin{align*}
[T(MH), T(KH)-T(KA)]&=[T(\Delta Z), T(KZ)-T(KA)]\\
&=[T(EH), T(KH)-T(K\Lambda)]
\end{align*}
entonces, $$[T(MH), T(EH)]=[T(KH)-T(KA), T(KH)-T(K\Lambda)],$$ de lo que se sigue que
\begin{align*}
[T(MH)-T(EH), T(MH)]&=[T(K\Lambda)-T(KA), T(KH)-T(KA)],\\
\intertext{entonces}\\
[T(\Theta E), T(K\Lambda)-T(KA)]&=[T(MH), T(KH)-T(KA)]
\end{align*}
La proporción $$[T(\Theta E), T(KE)-T(K\Delta)]$$ entonces está compuesta de las proporciones constantes $$[T(MH), T(KH)-T(KA)]\,\text{y}\, [T(K\Lambda)-T(KA), T(KE)-T(K\Delta)]$$ i.e., de la proporción para el cuadrado sobre la ordenada y el rectángulo sobre la abscisa de la sección meridiana y la proporción $[T(KA), T(K\Delta)],$ que no es una proporción $1:1.$ El lugar de \textLipsias{J} entonces es un amblitoma, pero este no es similar a la sección meridiana.
\end{proof}
\begin{propo}
Si un ambliconoide (\textLipsias{\>{a}mblikwnoide}\,$=\text{hiperboloide de revolución}$) es cortado por un plano que encuentra todos los generadores del cono envolvente del conoide (\textLipsias{konoide}) y este no forma ángulo recto con el eje, la sección será un oxitoma (\textLipsias{\>{o}xitoma}), y el diámetro más grande será la parte del corte externo en el diámetro del conoide de la intersección para el plano de corte y el plano a través del eje en ángulo recto a los planos de corte.
\end{propo}
\begin{proof}\footnote{Cf. Dijksterhuis. Op. cit. pp. 115.}
De la figura \ref{B.5}, encontramos que
\begin{figure}[ht!]
\begin{center}
  \includegraphics[scale=0.6]{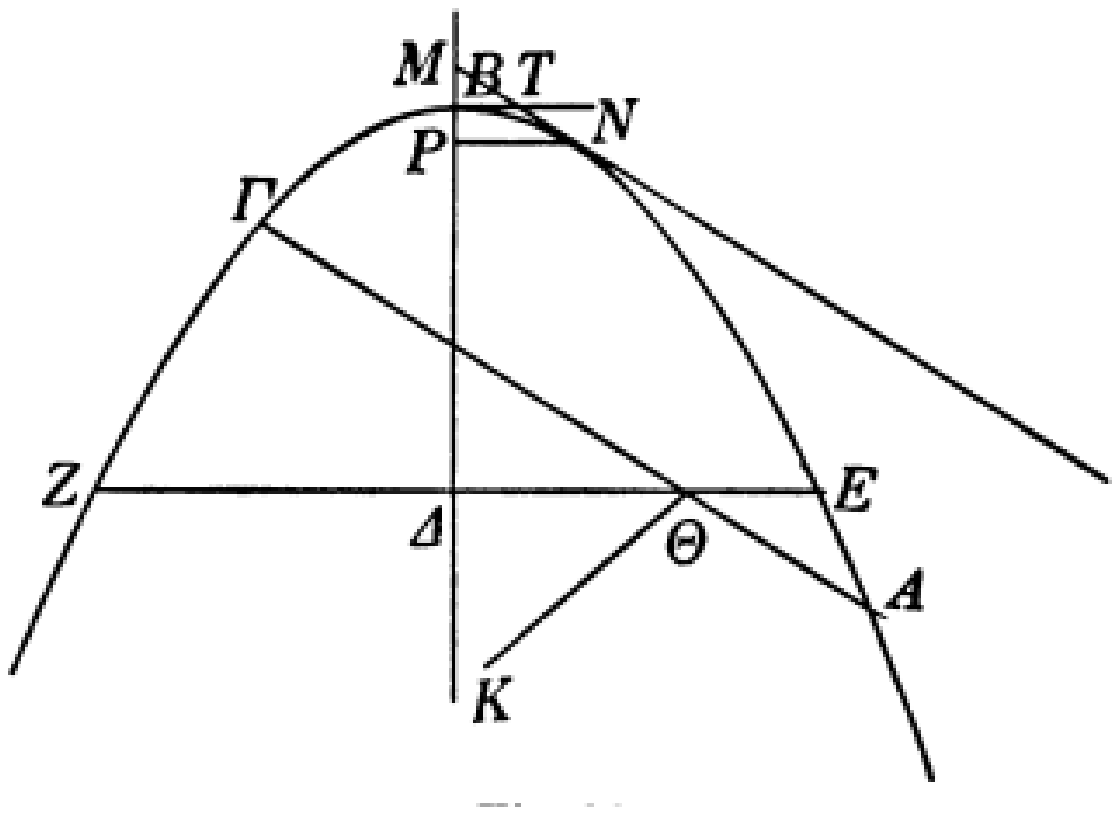}\\
  \caption{}\label{}
\end{center}
\end{figure}
\newline
$$[T(K\Theta), O(A\Theta, \Gamma\Theta)]=[T(BT), T(NT)]$$ del que ya se sigue que el lugar de \textLipsias{K} es un oxitoma. Ahora, $$BP>BM,\,\therefore\,TN>TM>BT,$$ por tanto, $BT<NT.$ Entonces \textLipsias{AG} es el diámetro más grande.
\end{proof}
\section{Superficies cuádricas, el tratamiento algebraico}
\subsection{Los primeros métodos}
El primer texto que recoge un estudio algebraico sistemático para cuádricas en general pertenece a Euler. Él no cubrió muchos hechos, pero forjó una herramienta para estudiar tales superficies en detalle. Su técnica esencial fue el cambio de ejes rectangulares. Iniciemos con la ecuación general de segundo grado en dos variables: $$\alpha z^2+\beta yz+\gamma xz+\delta y^2+\epsilon xy+\zeta x^2+\eta x+\theta y+\iota z+\kappa=0.$$ Observemos las partes de la superficie. A menos que vivan enteramente en un dominio infinito, necesita ser posible para a lo sumo una de las variables tomar condiciones infinitas; sea entonces la variables $z$ quien las tome.\\ En comparación, nosotros podemos dejar de lado la constante y los términos lineales, así que en el infinito la superficie actúa similarmente como el cono $$\alpha z^2+\beta yz+\gamma xz+\delta y^2+\epsilon xy+\zeta x^2=0.$$ Este cono no tendrá puntos reales excepto el origen.\\ Para esto es necesario que $$4\alpha\delta-\beta^2>0;\,4\alpha\zeta-\gamma^2>0;\,4\delta\zeta-\epsilon^2>0.$$ Tales condiciones necesarias no son suficientes. Necesita tenerse que para todas las condiciones de $x$ y $y$
\begin{align*}
4\alpha\left(\delta y^2+\epsilon xy+\zeta x^2\right)>{\left(\beta y+\gamma x\right)}^2,&\\
\left(4\alpha\delta-\beta^2\right)y^2+2\left(2\alpha\epsilon-\beta\gamma\right)xy+\left(4\alpha\zeta-\gamma^2\right)x^2>0&.
\end{align*}
Esto realmente no es una dirección muy satisfactoria para dicho problema.\\ El próximo paso dado por Euler consistía en simplificar los términos cuadráticos. Aquí él establece sin una demostración adecuada que por una rotación de los ejes nosotros podemos eliminar los términos mixtos en $xy, yz, zx$. La ecuación resultante es
\begin{equation}
Ap^2+Bq^2+Cr^2+Gp+Hq+Ir+K=0
\end{equation}
Cuando $ABC\neq 0$ nosotros podemos movernos a los ejes paralelos, así que $$Ap^2+Bq^2+Cr^2+L=0.$$ Revirtiendo la posibilidad de que un término cuadrático, por decir $z^2$, tendría que anularse, él escribe $$Ap^2+Bq^2+Cp+Hq+Ir+L=0.$$ Él entonces intercambia los ejes paralelos, obteniendo $$Ap^2+Bq^2=ar.$$ El tratamiento temprano obtuvo un progreso sustancial en los trabajos de Monge-Hachette. Monge fue uno de los maestros más grandes del mundo, y sabemos que muchas de sus grandes ideas estaban contenidas en las lecturas que él deliberó en la \textit{École Polytechnique}. La ecuación estándar escrita como
\begin{equation}\label{2}
At^2+A'u^2+A''v^2+2Buv+2B'vt+2B''ut+2Ct+2C'u+2C''v=k
\end{equation}
una línea a través de $(t',u',v')$ es escrita como $$t-t'=l(v-v')\quad u-u'=m(v-v').$$ Asumiendo que $(t',u',v')$ no está en la superficie, sustituimos $t$ por $u$ y tenemos una ecuación que es lineal en $v$. Tomemos la raíz de $v=v'$ y substituyendo por $l$ y tomando a $m$, tenemos el plano tangente
\begin{align*}
&t\left(At'+B''u'+B'v'+C\right)+u\left(B''t'+A'u'+Bv'+C\right)+\\
&+v\left(B't'+Bu'+A''v'+C''\right)+Ct'+Cu'+C'''v'-K=0
\end{align*}
Los autores ahora van a investigar el centro de la superficie. Ellos substituyen en (\ref{2}) $$t=x+\alpha,\,u=y+\beta,\,v=z+\gamma.$$ Los coeficientes de $2x, 2y, 2z$ son $$A\alpha+B''\beta+B'\gamma+C,\,B''\alpha+A'\beta+B\gamma+C',\,B'\alpha+B\beta+A''\gamma+C''.$$ Esta primera clase de superficie, que son los cono con centro, son aquellas en las que tales tres expresiones pueden ser hechas para que se anulen. En nuestra notación $$\left|\begin{array}{ccc}A&B''&B'\\B''&A'&B\\B'&B&A''\end{array}\right|\neq 0.$$ Nosotros entonces tenemos la forma reducida
\begin{equation}\label{3}
Ax^2+A'y^2+A''z^2+2Byz+2B'zx+2B''xy=H
\end{equation}
El problema siguiente es más complicado y tiene que ver con los ejes de tal rotación en que los términos del producto se anulan.\\ La ecuación del plano tangente es $$\left[Ax'+B''y'+B'z'\right]x+\left[B''x'+A'y'+Bz'\right]y+\left[B'x'+By'+A''z'\right]z=H.$$
Esto también puede ser escrito como $$L(x-x')+M(y-y')+N(z-z')=0.$$ Si la normal va através del centro $$\dfrac{L}{x'}=\dfrac{M}{y'}=\dfrac{N}{z'}.$$ Los autores Monge-Hachette realizan una larga serie de manipulaciones, y finalmente muestran que la determinación de aquellas normales depende de la ecucación cúbica
\begin{equation}\label{4}
s^3\left(AB^2+A'{B'}^2+A''{B''}^2-2AA'A''-2BB'B''\right)+s^2\left(A'A''+A''A+AA'-B^2-2{B'}^2\right)+s\left(A+A'+A''\right)+1=0
\end{equation}
Nosotros podemos observar que los autores no trataron los coeficientes aquí invariantes por un cambio rectangular de ejes. Dejando ésta ecuación por un momento, regresando a \eqref{2}, y resolviendo para $t,u,v$ se tiene
\begin{align*}
\dfrac{1}{2}t&=-\dfrac{B'v+B''u+C}{A}+\sqrt{T};\\
\dfrac{1}{2}u&=-\dfrac{B''t+Bv+C'}{A'}+\sqrt{U};\\
\dfrac{1}{2}v&=-\dfrac{Bu+B'v+C''}{A''}+\sqrt{V}.
\end{align*}
Los puntos medios para las cuerdas perpendiculares a $u,v$ están en $$At+B'v+B''u+C=0.$$ Este plano va muy cerca del centro. Observemos entonces que los puntos medios de algún conjunto de cuerdas paralelas están en el plano a través del centro. Posteriormente, observamos de las propiedads de las cuerdas paralelas de la cónica con centro que tenemos conjuntos infinitos de tres diámetros donde el plano de cada par bisecta todas las cuerda paralelas al tercero, y es paralela al plano en la extremidad del tercero.\\ Regresando a la ecuación (\ref{4}) que es cúbica y real, observamos que necesita tener una solución real, dado un plano perpendicular a las cuerdas que esta bisecta. Los ejes principales en este plano tienen otros diámetros que son el mismo.\\ Tomando aquellos ejes principales como ejes coordenados, tenemos la forma canónica de la ecuación de una cónica con centro
\begin{equation}\label{5}
\dfrac{x^2}{a^2}\pm\dfrac{y^2}{b^2}\pm\dfrac{z^2}{c^2}=1
\end{equation}
Monge-Hachette obtienen la ecuación típica de un paraboloide $$Px^2+P'y^2-4PP'x=0.$$ Ellos realizan algunos esfuerzos para determinar cuándo una cuádrica será una superficie de revolución.\\ La simple ecuación (\ref{5}) es útil para determinar nuevas propiedades de la cuádrica con centro. Aquellas son desarrolladas en \textit{Monge-Hachette}, pp. 194 y ss.\\ Si nosotros tenemos dos cónicas en un plano de tal naturaleza que los coeficientes de los términos cuadráticos sean proporcionales, podemos escribirlos como:
\begin{align*}
Ax^2+2Bxy+Cy^2+2Dx+2Ey+F&=0,\\
Ax^2+2Bxy+Cy^2+2D'x+2E'y+F'&=0.
\end{align*}
Pero el primero de aquellos tendría que ser cambiado por una translación de los ejes en la forma $$Ax^2+2Bxy+Cy^2+F_1=0,$$ mientras que una translación similar de los ejes tendría que cambiar al otro por $$Ax^2+2Bxy+Cy^2+F_2=0,$$ y esto muestra que las curvas son similares, y están en el mismo lugar. Iniciemos a continuación con la ecuación (\ref{3}) y cortemos la superfice por
\begin{align*}
&z=Lx+My+N\\
&x^2\left(A+A''L^2+2B'L\right)+2xy\left(B''+BL+B'M+A''LM\right)+y^2\left(A'+A''M^2+2BM\right)+Px+Qy+R=0.
\end{align*}
Los tres primeros coeficientes son independientes de $N$. Esto produce:
\begin{teor}
Las secciones paralelas de una cuádrica son cónicas similares,\footnote{Cf. la definición dada por el magno Arquimedes en la p.8 de este apéndice.} y están en el mismo lugar.\quad\qedsymbol
\end{teor}
Considere en particular el elipsoide $$\dfrac{x^2}{a^2}+\dfrac{y^2}{b^2}+\dfrac{z^2}{c^2}=1\quad(a>b>c)$$ y corte este por la esfera $$x^2+y^2+z^2=b^2$$
$$\left(\dfrac{x\sqrt{(a^2-b^2)}}{a}+\dfrac{2\sqrt{(b^2-c^2)}}{c}\right)\left(\dfrac{x\sqrt{(a^2-b^2)}}{a}-\dfrac{z\sqrt{(b^2-c^2)}}{c}\right)=0.$$
La sección del elipsoide y la esfera son dos círculos. Un tratamiento similar puede ser aplicado a las otras cuádricas con centro.
\begin{teor}
Una cuádrica con centro que no es una superficie de revolución tendrá dos conjuntos de secciones circulares en planos paralelos.\\ Los centros de los círculos de tales conjuntos estarán sobre un diámetro.\quad\qedsymbol
\end{teor}
La primera edición de Monge-Hachette apareció en 1802 y contenía lo que podría observarse como una mención temprana de las secciones circulares de una cuádrica. Debe decirse, sin embargo, que sobre la p. 163 de la obra de D'Alembert\footnote{Jean Le Rond D'Alembert, \textit{Opuscules mathématiques}, Vol. vii. París, 1780, pp. 171.} está el estamento no demostrado de que el elipsoide, que él llama un \textit{esferoide}, tendrá una sección circular.\\ Podemos escribir la ecuación del hiperboloide de una hoja en la forma $$Px^2+P'y^2-P''z^2=1,$$ y cortar este por el plano $$y=\alpha x +\beta.$$ La proyección de la intersección sobre el plano $x,z$ es $$(P+\alpha^2P')x^2-P''z^2+2P'\alpha\beta x+P'\beta^1-1=0.$$ Elejimos $\alpha$ y $\beta$ de tal forma que $$\beta^2P-\alpha^2=\dfrac{P}{P'},\,\beta=\sqrt{\left(\dfrac{\alpha}{p}+\dfrac{1}{P'}\right)};$$ Entonces la proyección es $${\left(x\sqrt{(P+\alpha^2P')}+\alpha\sqrt{\dfrac{P'}{P}}\right)}^2-{\left(z\sqrt{P''}\right)}^2=0.$$ Esto es un par de líneas rectas. El paraboloide hiperbólico puede ser tratado en la misma dirección.
\begin{teor}
El hiperboloide de una hoja y el paraboloide hiperbólico pueden ser generados en dos formas diferentes por una línea que siempre está en movimiento para intersectar las tres líneas mutuamente asimétricas.\quad\qedsymbol
\end{teor}
El estamento temprano para este teorema puede ser hallado en Monge,\footnote{\textit{Sur les lignes de courbure de la surface de l'ellipsoide}. Journal de L'École Polytechnique. Vol. i. 1794} p. 5, sin demostración.\\ Suficientemente curioso es el hecho de que éste teorema fue demostrado mucho antes para el hiperboloide de revolución, por Wren,\footnote{Sir Christopher Wren. \textit{The generation of a hyperbolical cylindroid.} Philosophical Transactions of the Royal Society. Vol. iv. 1669. (172).} p. 333.\\
\begin{figure}[ht!]
\begin{center}
\includegraphics[scale=0.5]{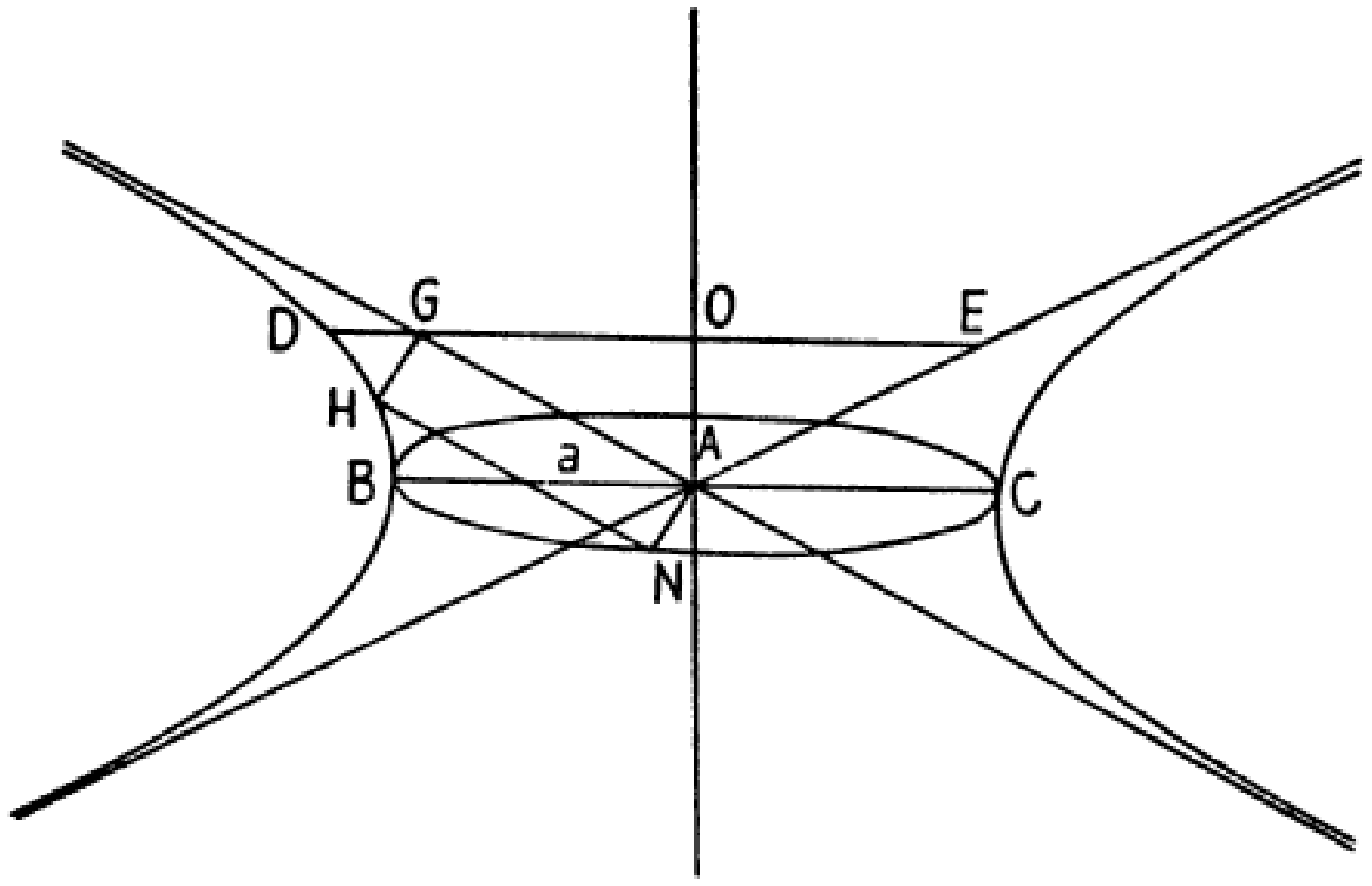}\\
\caption{}\label{B.7}
\end{center}
\end{figure}
\newline
Supóngase (ver Fig. \ref{B.7}) que la hipérbola en el plano del papel es rotada sobre el eje $OA.$ Sea $AN$ perpendicular al plano del papel. De $G$, algún punto sobre la asíntota, trace una línea paralela a $AN$ que tendrá que encontrar la superficie en $H$.
\begin{align*}
&HG^2=OH^2-OG^2=OD^2-OG^2,\\
&\dfrac{OD^2}{a^2}-\dfrac{OA^2}{b^2}=1;\,\dfrac{OG^2}{OA^2}=\dfrac{a^2}{b^2};\,\dfrac{OG^2}{a^2}=\dfrac{OA^2}{b^2}\\
&OD^2-OG^2=a^2=AN^2,\\
&GH=AN
\end{align*}
Entonces $NH$ es paralelo a $AG$, o a la línea a través de $N$ paralela a la asíntota que está completamente embebida en la superficie.
\begin{teor}[\textbf{Wren}]
La sección de un hiperboloide de revolución de una hoja en un plano a través de la asíntota de una hipérbola perpendicular al plano de su curva son dos líneas paralelas a su tangente.
\end{teor}
Se tienen otros dos teoremas que pueden ser hallados en la obra de Monge-Hachette.\\ Supongamos que cortamos una elipse de un elipsoide y tomamos esta pegándola a la curva de un cono, cuyo vértice es el final de un diámetro que contiene los centros de los círculos cortados de la superficie. Ahora se puede observar que este cono corta el elipsoide. Parte de esta sección es la elipse. El resto necesita ser una cónica, a través del vértice del cono, y así se tienen dos generadores (conos imaginarios). Pero como el plano tangente encuentra la superficie en el círculo, las dos líneas son minimales, y un plano paralelo cortará el cono en un círculo.
\begin{teor}
Si el vértice de un cono es un punto sobre un elipsoide que está al final de un diámetro que contiene los centros de un conjunto de secciones circulares, y los generadores son unidos por una elipse cortada del elipsoide, entonces aquellos planos que cortan los círculos de el elipsoide también cortarán círculos del cono.\quad\qedsymbol
\end{teor}
\begin{teor}
Si tres planos mutuamente perpendiculares son tangentes a un elipsoide, el lugar de intersección de tales puntos es una esfera con los mismos centros del elipsoide.
\end{teor}
\begin{proof}
Empleando la notación de Monge-Hachette, consideremos la superficie $$Px^2+P'y^2+P''z^2=1.$$ Sean los puntos de contacto $(x',y',z'), (x'',y'',z''), (x''',y''',z''').$ Se tiene que:
\begin{align*}
\textit{Tres ecuaciones del tipo}\,P{x'}^2+P'{y'}^2+P''{z'}^2&=1\\
\textit{Tres planos tangentes}\,P{x'x}^2+P'{y'y}^2+P''{z'z}^2&=1\\
\textit{Tres condiciones de perpendicularidad}\,P^2{x'x''}^2+{P'}^2{y'y''}^2+{P''}^2{z'z''}^2&=1\\
\end{align*}
Ahora se introducen 9 variables:
$$\begin{array}{cccc}
Px'=a,&Px''=a',&Px'''=a'',\\
P'y'=b,&P'y''=b',&P'y'''=b'',\\
P''z'=c,&P''z''=c',&P'''z'''=c''',\\
\sum a^2=R^2,&\sum {a'}^2={R'}^2,&\sum{a''}^2={R''}^2
\end{array}$$
Se tiene, entonces,
\begin{align*}
\textit{Tres ecuaciones}\,\dfrac{a^2}{P}+\dfrac{b^2}{P'}+\dfrac{c^2}{P''}&=1\\
\textit{Tres ecuaciones}\, ax+by+cz&=1\\
\textit{Tres ecuaciones}\, aa'+bb'+cc'&=0
\end{align*}
Se tiene la matriz
$$\left(\begin{array}{ccc}
\dfrac{a}{R}&\dfrac{b}{R}&\dfrac{c}{R}\\\\
\dfrac{a'}{R'}&\dfrac{b'}{R'}&\dfrac{c'}{R'}\\\\
\dfrac{a''}{R'''}&\dfrac{b''}{R'''}&\dfrac{c''}{R'''}
\end{array}\right)$$
que es ortogonal, y sean
\begin{align*}
x&=\left(\begin{array}{ccc}
\dfrac{1}{R}&\dfrac{b}{R}&\dfrac{c}{R}\\\\
\dfrac{1}{R'}&\dfrac{b}{R'}&\dfrac{c}{R'}\\\\
\dfrac{1}{R'''}&\dfrac{b}{R'''}&\dfrac{c}{R'''}
\end{array}\right)\\
y&=\left(\begin{array}{ccc}
\dfrac{a}{R}&\dfrac{1}{R}&\dfrac{c}{R}\\\\
\dfrac{a}{R'}&\dfrac{1}{R'}&\dfrac{c}{R'}\\\\
\dfrac{a}{R'''}&\dfrac{1}{R'''}&\dfrac{c}{R'''}
\end{array}\right)\\
z&=\left(\begin{array}{ccc}
\dfrac{a}{R}&\dfrac{b}{R}&\dfrac{1}{R}\\\\
\dfrac{a}{R'}&\dfrac{b}{R'}&\dfrac{1}{R'}\\\\
\dfrac{a}{R'''}&\dfrac{b}{R'''}&\dfrac{1}{R'''}
\end{array}\right)
\end{align*}
$$x^2+y^2+z^2=\dfrac{1}{R^2}+\dfrac{1}{{R'}^2}+\dfrac{1}{{R''}^2}$$
teniéndose exactamente que
\begin{align*}
\dfrac{1}{P}+\dfrac{1}{P'}+\dfrac{1}{P''}&=\dfrac{1}{R^2}+\dfrac{1}{{R'}^2}+\dfrac{1}{{R''}^2};\\
x^2+y^2+z^2&=\dfrac{1}{P}+\dfrac{1}{P'}+\dfrac{1}{P''}
\end{align*}
\end{proof}

\nocite{*}
\bibliographystyle{abbrvnat}
\bibliography{Archi}

\end{document}